\documentclass[12pt]{amsart}
\font\emailfont=cmtt10

\usepackage{amsmath,amsthm,amsfonts,amscd,flafter,epsf}
\usepackage{graphics,todonotes,appendix}
\usepackage{epsfig}
\usepackage{psfrag}
\usepackage[all,cmtip]{xy}
\usepackage{comment}

\newcommand\commentable[1]{#1}

\newtheorem{thm}{Theorem}

\newtheorem{question}[thm]{Question}
\newtheorem{theorem}{Theorem}[section]
\newtheorem{prop}[theorem]{Proposition}
\newtheorem{cor}[theorem]{Corollary}

\newtheorem{lemma}[theorem]{Lemma}

\newtheorem{defn}[theorem]{Definition}

\newtheorem{remark}[theorem]{Remark}

\def\endproofof{\relax\ifmmode\expandafter\endproofmath\else
  \unskip\nobreak\hfil\penalty50\hskip.75em\hbox{}\nobreak\hfil\bull
  {\parfillskip=0pt \finalhyphendemerits=0 \bigbreak}\fi}
\def\endproofofmath$${\eqno\bull$$\bigbreak}

\def\endproof{\relax\ifmmode\expandafter\endproofmath\else
  \unskip\nobreak\hfil\penalty50\hskip.75em\hbox{}\nobreak\hfil\bull
  {\parfillskip=0pt \finalhyphendemerits=0 \bigbreak}\fi}
\def\endproofmath$${\eqno\bull$$\bigbreak}
\def\bull{\vbox{\hrule\hbox{\vrule\kern3pt\vbox{\kern6pt}\kern3pt\vrule}\hrule}}

\newcommand{\Z}{\mathbb{Z}}

\newcommand\SpinC{\mathrm{Spin}^c}
\newcommand{\F}{\mathbb F}

\newcommand\relspinc{\underline{\spinc}}

\newcommand\Filt{\mathcal F}

\newcommand\ModSphere{\ModFlow\left({\mathbb S}\longrightarrow 
\Sym^{g-1}(\Sigma_{1})\times \Sym^2(\Sigma_{2})\right)}
\newcommand\ModSpheres\ModSphere
\newcommand\CF{CF}

\newcommand\CFa{\widehat{CF}}

\newcommand\HFa{\widehat{HF}}

\newcommand\UnparModSp{\widehat \ModSp}
\newcommand\UnparModFlow\UnparModSp
\newcommand\Mod\ModSp

\newcommand{\spinc}{\mathfrak s}

\newcommand\ModMaps{\mathcal M}
\newcommand\ModSp\ModMaps

\newcommand\alphas{\mbox{\boldmath$\alpha$}}

\newcommand\betas{\mbox{\boldmath$\beta$}}

\newcommand\spincrel\relspinc

\newcommand\Dual{\mathcal D}
\newcommand\Duality\Dual

\newcommand\ons{Ozsv{\'a}th and Szab{\'o}}
\newcommand\os{{Ozsv{\'a}th-Szab{\'o}}}

\newcommand\Sym{\mathrm{Sym}}

\newcommand\Wdual{W^\dagger}

\commentable{

\title[{On naturality of the Ozsv\'ath-Szab\'o contact invariant}] 
{On naturality of the Ozsv\'ath-Szab\'o contact invariant}

}

\author[Matthew Hedden]{Matthew Hedden}
\address{Department of
Mathematics, Michigan State University, MI \newline
\indent{\emailfont{mhedden@math.msu.edu}}}
\author[Lev Tovstopyat-Nelip]{Lev Tovstopyat-Nelip}
\address{Department of
Mathematics, Michigan State University, MI \newline
\indent{\emailfont{tovstopy@math.msu.edu}}}

\thanks{MH gratefully acknowledges support from NSF grant DMS-1709016. }

\begin{document}

\begin{abstract}  We discuss functoriality properties of the \os \  contact invariant, and expose a number of results which seemed destined for folklore.  We clarify the (in)dependence of the invariant on the basepoint, prove that it is functorial with respect to contactomorphisms, and show that it is strongly functorial under Stein cobordisms.    \end{abstract}

\maketitle
\section{Introduction}
Heegaard Floer homology  provides a seemingly ever-growing number of invariants for low-dimensional topology.  Its influence has perhaps most firmly been felt within the realm of $3$-dimensional contact geometry, upon which the \os\ contact invariant \cite{Contact} and its refinements  have had a profound impact.  In its most basic form, the contact invariant of a closed $3$-contact manifold $(Y,\xi)$ is an element residing in the Heegaard Floer homology group $\HFa(-Y)$ of the underlying manifold, equipped with opposite the orientation it receives from the contact structure (a group which, perhaps more naturally, can be identified with the Floer cohomology of $Y$).  A decade after its initial development by \ons\ \cite{HolDisk,HolDiskTwo}, Juhasz, Thurston, and Zemke  discovered  a subtle dependence of Heegaard Floer homology on a choice of basepoint underlying its definition \cite{JTZ}. Indeed, they showed that Floer homology cannot associate a well-defined group to a 3-manifold alone, but only to a 3-manifold equipped with a basepoint.    This raises the questions of whether the contact element is well-defined and how it depends on the basepoint,  questions raised but not pursued in \cite{JTZ}.  

The purpose of this article is to examine these questions, and further explore functoriality properties of the contact invariant.   As a first step, we show that the contact element is a well-defined element (as opposed to an orbit under the group of graded automorphisms) in the Floer homology of a  pointed contact 3-manifold, see Theorem \ref{well-defined} below.  To prove this, we first establish an appropriate definition and notion of equivalence for pointed contact 3-manifolds and incorporate these ideas into the Giroux correspondence.  We then revisit  \os's proof of invariance within the naturality framework of \cite{JTZ}, ensuring that Heegaard surfaces, links, open books, basepoints, etc. can be arranged to be explicitly embedded in a fixed 3-manifold.  

Having checked the aforementioned details, we turn to a refined understanding of invariance of the contact class, showing that it is functorial with respect to pointed contactomorphisms. 

\begin{thm}\label{thm:funcconintro}Suppose $f$ is a pointed contactomorphism  between pointed contact $3$-manifolds $(Y,\xi,w)$ and $(Y',\xi',w')$.  Then the  induced map on Floer  homology  \[f_*: \HFa(-Y,w)\rightarrow \HFa(-Y',w')\] carries $c(\xi,w)$ to $c(\xi',w')$.  
\end{thm} 

\noindent The functoriality above is an immediate consequence of the functoriality of Floer homology under pointed diffeomorphisms from \cite{JTZ}, provided one parses the Giroux Correspondence in a categorical framework, which we clarify with Proposition \ref{functGiroux}.  

We then show that, while the group in which the contact element lives depends on the basepoint, the contact element itself does not.  This can be explained as follows: the dependence of the Floer homology of $Y$ on the basepoint is determined by a functor \[HF(Y,-): \Pi_1(Y)\rightarrow \mathrm{iGrp}\] from the fundamental groupoid of $Y$ to the isomorphism subcategory of groups.  Concretely, this just means that there is a well-defined  isomorphism  between Floer groups $\HFa(Y,w)$ and $\HFa(Y,w')$ associated to a homotopy class of a path between  $w$ and $w'$, which is compatible with concatenation (see the next section for more details).   If one restricts to the subgroups of $\HFa(Y,w)$ spanned by contact classes, which we denote $cHF(Y,w)$, this functor yields a transitive system indexed by points in $Y$; that is, the isomorphisms $cHF(Y,w)\rightarrow cHF(Y,w')$ are independent of paths.  We can therefore consider the direct limit of the transitive system, which we call the {\em contact subgroup} of Floer homology, and denote $cHF(Y)$.

\begin{thm}\label{thm:basepointindependence} The contact subgroup $cHF(Y)$ is a well-defined invariant of an (unpointed) 3-manifold, functorial with respect to  diffeomorphisms.  There is an  element $c(\xi)\in cHF(Y)$, associated to a contact structure $\xi$ on $Y$, and the map $f_*:cHF(Y)\rightarrow cHF(Y')$ induced by a contactomorphism  $f:(Y,\xi)\rightarrow (Y',\xi')$ sends $c(\xi)$ to $c(\xi')$.
\end{thm} 

One can view this result in two ways.  On the one hand, it follows from the functoriality established in Theorem \ref{thm:funcconintro}, together with the fact that we can realize the change-of-basepoint diffeomorphism associated to a homotopy class of path by a contactomorphism (see Proposition \ref{contactpointpushing}).  On the other, it can be viewed as a consequence of  Zemke's calculation of the representation of the fundamental group on Floer homology in terms of the $H_1(Y)/\mathrm{Tor}$ action and the basepoint action $\Phi_w$, together with the fact that contact classes are in the kernel of the $H_1(Y)/\mathrm{Tor}$ action.  Adopting the latter perspective, we see that the contact subgroup is a subgroup of a larger basepoint independent subgroup, arising as the kernel of the $H_1(Y)/\mathrm{Tor}$ action.  Note that these considerations dash  any na{i}ve hope that Heegaard Floer homology is generated by contact classes, much less to elements associated to taut foliations, and indicate that such a conjecture might more reasonably be made in the context of a twisted coefficient system in which the $H_1(Y)/\mathrm{Tor}$ action vanishes, e.g. totally twisted coefficients.

Having clarified the definition and invariance properties of the contact element, we  then show that it is functorial under Stein cobordisms in a  precise way.

\begin{thm}\label{thm:naturalityintro}  Suppose $(W,J,\phi)$ is a Stein cobordism from a contact 3-manifold $(Y_1,\xi_1)$ to a contact 3-manifold $(Y_2,\xi_2)$.  Then \[F_{\Wdual,\mathfrak{k}}(c(Y_2,\xi_2))=c(Y_1,\xi_1),\]
where  $W$ is viewed as a cobordism from $-Y_2$ to $-Y_1$, and $\mathfrak{k}$  is the canonical $\SpinC$ structure associated to $J$.  Moreover 
\[F_{\Wdual,\spinc}(c(Y_2,\xi_2))=0\]
for $\spinc\ne \mathfrak{k}$.
\end{thm}

In a weaker form,  such a result follows fairly easily from the existing literature, and was widely known to experts. See Section \ref{sec:Stein} for a discussion. In the present level of specificity, the proof is slightly more involved than one might initially expect, owing largely to the nature of the composition law for cobordism maps in Heegaard Floer theory.     We remark that the incoming and outgoing boundaries of  $W$ are not assumed to be connected, and that  Theorem \ref{thm:naturalityintro} immediately yields a generalization of Plamenevskaya's independence result for contact invariants from \cite{OlgaDistinct}; see Corollary \ref{olgageneralization}.

It would be interesting to know how much naturality of the contact element persists as one weakens assumptions on the cobordism.  It is known, for instance, that the contact element in monopole Floer homology is natural under strong symplectic cobordisms \cite[Theorem 1]{Echeverria}; see also \cite{MrowkaRollin}.  One would  thus expect an affirmative answer to the following:
\begin{question} Is the contact element natural under strong symplectic cobordisms?
\end{question}

Finally, we note that we work with Heegaard Floer homology with $\Z/2\Z$ coefficients throughout.  The naturality results of \cite{JTZ} have been extended to projective $\Z$ coefficients (i.e., $\Z/\pm1$) \cite{Gartner}, but at the moment these extensions have not been established for the graph cobordism maps.  Assuming they will be, the results at hand should immediately extend to the refined setting.  
\\

\noindent{\bf Acknowledgements:} This article stemmed out of the first author's work with Katherine Raoux, discussed at the BIRS workshop ``Interactions of gauge theory with contact and symplectic topology in dimensions 3 and 4", and a particular application of that work which required naturality of the contact class under Stein cobordisms.  It is our pleasure to thank Katherine, BIRS, and the workshop organizers for  inspiring us to write this article. 
\section{Functoriality of the contact class under diffeomorphisms}\label{sec:functoriality}
In this section we clarify the dependence of the contact class on the basepoint used in the definition of Heegaard Floer homology, and highlight how the results of Juhasz-Thurston-Zemke \cite{JTZ} and Zemke \cite{ZemkeGraphCob} couple with \os's argument from \cite{Contact} to imply that the contact class is a well-defined invariant of a {\em pointed} contact 3-manifold, up to pointed isotopy (Theorem \ref{well-defined}).   This invariant is shown to be functorial under pointed contactomorphisms (Theorem \ref{thm:funccon}).  We then show that while, according to \cite{JTZ}, the Heegaard Floer homology group in which the contact element lives depends in an essential way upon the basepoint, the contact invariant is essentially independent from the basepoint (Theorem \ref{thm:basepointindependence}).   Unless otherwise specified, all 3-manifolds are assumed to be compact and oriented, and contact structures assumed to be cooriented.

Recall from \cite{JTZ} that the Heegaard Floer homology group of a pointed 3-manifold $(Y,w)$ is defined as the direct limit of a transitive system of groups and isomorphisms defined by pointed Heegaard diagrams $(\Sigma,\alphas,\betas,w)$ embedded in $(Y,w)$ and pointed Heegaard moves passing between them (together with auxiliary choices of almost complex structures).  See \cite[Theorem 1.5] {JTZ}, and the surrounding discussion.  The contact invariant should therefore be interpreted as an element in the aforementioned direct limit.  As such, it would appear to depend on the basepoint, and we therefore make the following definition  

\begin{defn} A {\em pointed contact 3-manifold}  is a 3-manifold $Y$ equipped with a contact structure $\xi$ and a distinguished basepoint $w$.  
\end{defn}

By Gray's theorem, an isotopy between contact structures is induced by an isotopy of the underlying (compact) 3-manifold, and two contact structures on $(Y,w)$ will be considered equivalent if we can find such an isotopy fixing the basepoint $w$.  

\begin{remark} One could consider a more  restrictive definition of equivalence where the isotopy fixes the contact plane at $w$. This  differs from the present notion only by the choice of oriented plane at $w$, a choice parametrized by a 2-sphere, and would have no effect on our results.  See  \cite[Lemma 2.45]{JTZ} for more details.
\end{remark}

In \cite{Contact}, \os\ defined an invariant of contact structures utilizing the Giroux correspondence between isotopy classes of contact structures on $Y$ and equivalence classes of fibered links in $Y$ under Hopf plumbing. Given a fibered knot representing $\xi$, its knot Floer homology has a distinguished filtered subcomplex in bottommost Alexander grading, whose homology is rank one \cite[Theorem 1.1]{Contact}.  Inclusion of this subcomplex in $\CFa(-Y)$ defines an element $c(\xi)\in \HFa(-Y)$ \cite[Definition 1.2]{Contact}, which they showed does not depend on the particular choice of fibered knot representing $\xi$ \cite[Theorem 1.3]{Contact}.  Absent from the the literature at that time, however, was an understanding of the dependence of the Floer homology group in which $c(\xi)$ resides on the choice of basepoint.  We restate their theorem so that this dependence is explicit, and outline the elements of their proof of invariance which should be refined accordingly.

\begin{theorem}\cite[Theorem 1.3]{Contact}\label{well-defined} Suppose two contact structures $\xi,\eta$ on the pointed 3-manifold $(Y,w)$ are equivalent.  Then $c(\xi,w)=c(\eta,w)\in \HFa(-Y,w)$.
\end{theorem}
\begin{proof}   We begin by observing that the  Giroux correspondence \cite{Giroux,EtnyreClay} has the following pointed analogue:
\[\frac{\{\mathrm{(pointed)\ open \ books\ }(L,\pi_L)\mathrm{\ in \ }{(Y^3,w)}\}}{\{\mathrm{(pointed)\ isotopy\ and\ positive\ Hopf\ plumbing}\}} = \frac{\{\mathrm{contact\ structures\ on}\ (Y^3,w)\}}{\{\mathrm{isotopy\ fixing\ w}\}}\]
where we emphasize that on the left-hand side we are considering {\em concrete} open  book decompositions of $Y$, by which we mean an embedded link $L\subset Y$ together with a fibration on its exterior $\pi_L: Y\setminus L\rightarrow S^1$ for which the boundary of the closure of each fiber is $L$. In these terms a pointed open book for the pointed contact manifold is an open book supporting $\xi$ for which the basepoint is contained in $L$.  Two open books are considered equivalent if they differ by a sequence consisting of ambient isotopies of  links  and ambient (de)plumbings with positive Hopf bands, where isotopies and (de)plumbings are required to fix the basepoint.  The pointed statement follows easily from the unpointed statement.  We remark that the inclusion of ambient isotopies of the open book is essential, though typically omitted or implicit in the literature.

\os's  proof  relies on two lemmas. If we denote the element associated to a fibered knot $K\subset Y$   by $c(K)\in \HFa(-Y)$, then \cite[Lemma 4.1]{Contact} states that this element is unchanged under connected summing with the right-handed trefoil, $T$; that is,  $c(K\#T)=c(K)$ for any fibered knot $K\subset Y$.  Then \cite[Lemma 4.4]{Contact} shows that the element associated to a fibered knot obtained by plumbing $2h$ right-handed Hopf bands to $K\subset Y$ is independent of the choice of plumbings.  Using plumbings which realize iterated connected sums with $T$, the result follows.   

The proof of the latter lemma goes by realizing the element associated to any genus $h$ stabilization as the image of fixed class in $\HFa(-Y)$ under a map induced by a cobordism $W$ which is diffeomorphic to $Y\times[0,1]$.  To do this, one observes that a genus $h$ stabilization can be obtained by attaching cancelling 4-dimensional 1- and 2-handles to $Y\times I$, where the former add handles to the page of the open book and the latter enact Dehn twists to the monodromy.  One then uses \cite[Theorem 4.2]{Contact}, which states that the element is carried naturally under the 2-handle cobordisms that add left-handed Dehn twists to the monodromy. 

According to \cite{JTZ}, one should refine these arguments so that they use embedded Heegaard diagrams in $Y$, with the basepoint lying on the embedded Heegaard surface.  When appealing to the functoriality with respect to cobordisms used in the proof of \os's second lemma, one must also be careful about the embedded path in $W\cong Y\times[0,1]$ from the basepoint to itself.

To address the first issue, consider a concrete pointed open book $(K,\pi_K)$ in $Y$ supporting (a contact structure isotopic to) $\xi$, with connected binding.  From this, one can construct \os's  Heegaard diagram adapted to $K$ from \cite[Section 3]{Contact}.  This construction can be done so that the diagram is embedded in $Y$, built from the union of the closure of two fibers by a stabilization, and so that it contains the basepoint.  The proof of naturality for knot Floer homology adapts to produce a functorial invariant from the category of pointed knots to a category whose objects are transitive systems of $\Z$-filtered complexes (where the maps in such systems are certain canonical filtered homotopy classes of filtered homotopy equivalences).   See \cite[Proposition 2.3]{Involutive} for the adaptation of the proof of naturality in  \cite{JTZ}  to the context of  transitive systems of complexes, and \cite[Proposition 2.8]{Adjunction} for a  discussion on how to apply this to   knots.
 Naturality implies that the generator of the homology of the bottommost filtered subcomplex defined by the embedded Heegaard diagram for the knot  $(K,\pi_K)$ produces a well-defined invariant $c(K,w)\in\HFa(-Y,w)$, by consideration of the inclusion-induced map. 
 
 To argue that the element $c(K,w)$ is invariant under connected summing with a trefoil, we observe that one can form the connected sum of the open book with the trefoil knot 
 {\em ambiently}, by embedding the trefoil and its fiber surface in a small ball near the basepoint, but in the complement of the Heegaard surface for $K$.  One can then form an embedded Heegaard diagram adapted to $K\#T$, which is a connected sum of embedded diagrams.  The K{\"u}nneth theorem for the  knot Floer homology of a connected sum \cite[Theorem 7.1]{Knots}, together with the fact that the new diagram is obtained from the initial diagram by a sequence of pointed embedded Heegaard moves, shows that   $c(K\#T,w)=c(K,w)$. 
 
The second lemma from \os's proof also goes through in the context of pointed 3-manifolds and concrete open books. Indeed,   if we are given a concrete pointed open book $(K',\pi')$  for $(Y,w)$ which is obtained from $(K,\pi)$ by ambiently plumbing $2h$ positive Hopf bands, we cancel the additional right-handed Dehn twists in the monodromy by   attaching $2h$ 4-dimensional 2-handles along curves in the page.  This results in a cobordism whose outgoing boundary is diffeomorphic to $Y\#^{2h}S^1\times S^2$.  Further attaching $2h$ 4-dimensional 3-handles to cancel the 2-handles results in a composite cobordism which is diffeomorphic, rel boundary, to $Y\times[0,1]$.  Moreover, since   the attaching regions for the $2$- and $3$- handles lie in the complement of the basepoint, the path traced by the basepoint in the cobordism is sent, under  a diffeomorphism to $Y\times[0,1]$, to the trivial path $w\times[0,1]$.   We can reverse the orientation of $Y$ before performing the aforementioned handle attachments, and the resulting map on the Floer homology of $\HFa(-Y,w)$ is the identity.   The result follows as in \cite{Contact} by appealing to the naturality of the contact invariant under addition of Dehn twists. \end{proof}

\begin{remark}\label{rem:HKMinvariance}  One could also establish pointed invariance of the contact element using the Honda-Kazez-Mati{\'c} perspective \cite{HKMcontact}.  Doing so would likely  establish naturality for the various contact invariants one can define using partial open books for 3-manifolds with convex boundary \cite{HKMcontactsutured}.
\end{remark}

To understand the functoriality of the contact class, we observe that pointed contact 3-manifolds form the objects of a category whose  morphisms are pointed isotopy classes of contactomorphisms.  With respect to this structure, the (pointed) Giroux correspondence is functorial.  To understand this, we make the following definition:

\begin{defn}  A {\em (pointed) diffeomorphism between concrete (pointed) open books} is an orientation-preserving diffeomorphism of pairs $f: (Y,L)\rightarrow (Y',L')$ which intertwines the fibrations on the link complements i.e.  $  \pi_L = \pi_{L'}\circ f_{Y\setminus L}$ (and which maps the  basepoint on $L$ to the basepoint on $L'$).
\end{defn}

If the open book $(L,\pi_L)$ supports a contact structure $\xi_L$ on $Y$, then a diffeomorphic open book $(L',\pi_L')$ supports a contact structure $\xi_{L'}$ on $Y'$ satisfying $f_*(\xi_L)=\xi_{L'}$.  Since the contact structure induced by an open book is only well-defined up to isotopy, we will regard diffeomorphisms of  open books up to  isotopy without loss of information.  In this way, a diffeomorphism of open books defines an isotopy class of contactomorphisms.


Conversely, given an isotopy class of contactomorphisms $f:(Y,\xi)\rightarrow(Y',\xi')$, we can push-forward an open book $(L,\pi_L)$ supporting $\xi$ under $f$,  yielding an open book $(f(L),\pi_L\circ f^{-1})$ supporting $(Y',\xi')$. The evident diffeomorphism of open books induces the given contactomorphism, up to isotopy.  In this way, the Giroux correspondence can be lifted to an isomorphism of  categories:

\begin{prop}[Functorial Giroux Correspondence]\label{functGiroux} There is an isomorphism of categories between the category of (pointed) contact 3-manifolds and (pointed) isotopy classes of contactomorphisms and the {\em concrete open book category}, whose objects are 3-manifolds equipped with concrete (pointed) open books up to ambient (pointed) Hopf plumbing and whose morphisms are (pointed) isotopy classes of (pointed) diffeomorphisms between open books.  
\end{prop}
\begin{proof} The (pointed) Giroux correspondence yields a bijection between objects which, in one direction, sends an (equivalence class of (pointed)) open book to the ((pointed) isotopy class of a) contact structure supporting it.  The discussion above shows that there are corresponding bijections between morphism sets.
\end{proof}

Using this, we can show that the contact class is functorial with respect to pointed contactomorphisms, Theorem \ref{thm:funcconintro} from the introduction.

\begin{theorem}[Functoriality under contactomorphisms] 
\label{thm:funccon}Suppose $f$ is a pointed contactomorphism  between pointed contact manifolds $(Y,\xi,w)$ and $(Y',\xi',w')$.  Then the map on Floer  homology  $f_*: \HFa(-Y,w)\rightarrow \HFa(-Y',w')$ carries $c(\xi,w)$ to $c(\xi',w')$.  
\end{theorem}
\begin{proof} This follows easily from the definition of the map on Floer homology associated to a pointed diffeomorphism, as described in \cite[Section 2.5, Definition 2.42]{JTZ}, together with the functorial Giroux correspondence.  More precisely, according to Section 2.5 of \cite{JTZ}, the map between Floer homology groups associated to a pointed isotopy class of diffeomorphism is defined by the map on transitive systems induced by  pushing forward embedded pointed Heegaard diagrams in $(Y,w)$ (and moves between them) to $(Y',w')$.  A pointed Heegaard diagram adapted to a concrete open book supporting $(Y,\xi)$ is mapped, via $f$, to a pointed Heegaard diagram adapted to a diffeomorphic concrete open book supporting $(Y',\xi')$.   Taking homology of these complexes gives rise to representatives for the direct limit of the transitive systems that define $\HFa(-Y,w)$ and $\HFa(-Y',w')$, respectively.    Under the induced map, the cycle representing the contact element for $c(\xi,w)$ is taken to that representing $c(\xi',w')$.  The result follows.
\end{proof}

Since the hat Floer homology groups depend on the basepoint, the above refinements are necessary in order to understand the invariance of the contact class. Having addressed this, however, we will now show that the contact class is essentially {\em independent} of the basepoint, relying on it only insomuch as it is required to define the group in which the class resides.  

  To explain this, recall that the map  Diff$(Y)\overset{\mathrm{ev}_w}\longrightarrow Y$ which evaluates a diffeomorphism at a basepoint is a Serre fibration, and the fiber over $w$ is  the pointed diffeomorphism group Diff$(Y,w)$.  The associated long exact sequence on homotopy shows that the pointed mapping class group of $Y$ is an extension of the mapping class group by $\pi_1(Y,w)$.   Concretely, this implies that if a pointed diffeomorphism is (unpointed) isotopic to the identity, then it is isotopic to a ``point-pushing map" about a loop representing an element in $\pi_1(Y,w)$.   If one considers instead the fiber of ev$_w$ over a different basepoint, $w'$, we see that any diffeomorphism of $Y$ sending $w$ to $w'$ which is (unpointed) isotopic to the identity is  isotopic through diffeomorphisms sending $w$ to $w'$, to a point-pushing map defined by a choice of arc $\gamma$ from $w$ to $w'$.  Moreover, any two such diffeomorphisms differ, up to pointed isotopy, by a point-pushing map along an element in $\pi_1(Y,w)$.   In light of this, the dependence of Floer homology on the basepoint is captured by a representation $\pi_1(Y,w)\rightarrow $Aut$(\HFa(Y,w))$ defined by isomorphisms associated to isotopy classes of point-pushing diffeomorphisms.  While this representation can be non-trivial, the following proposition  implies that it acts trivially on the subspace spanned by contact elements.

\begin{prop}  \label{prop:istpy}
Suppose that contact structures $\xi$ and $\eta$  on the pointed 3-manifold $(Y,w)$ are isotopic, induced by an isotopy of $Y$ which does not necessarily fix the basepoint.  Then $c(\xi,w)=c(\eta,w)\in \HFa(-Y,w)$
\end{prop}
\noindent We note that the functoriality of the contact invariant only implies that $f_*(c(\xi,w))=c(\eta,w)$, where $f$ is the endpoint of the isotopy.

\begin{proof}  Let $\phi_t:Y\times[0,1]\rightarrow Y$ denote the isotopy carrying $\xi$ to $\eta$, where $\phi_0=$Id$_Y$ and $\phi_1=f$ is a diffeomorphism fixing $w$, but where $\phi_t$ may not fix the basepoint for $0<t<1$.  The discussion preceding the proposition indicates that $f$ is isotopic to point-pushing map along a curve $\gamma$ representing an element $[\gamma]\in \pi_1(Y,w)$.  More precisely, a loop $\gamma$ based at $w$ can be regarded as an isotopy of embeddings of a point into $Y$ which, by the isotopy extension theorem, can be extended to an isotopy of $Y$ which is the identity outside a neighborhood of the image of $\gamma$.  The endpoint of this latter isotopy is a pointed diffeomorphism $f_\gamma: (Y,w)\rightarrow (Y,w)$ whose pointed isotopy class depends only on the homotopy class  $[\gamma]\in \pi_1(Y,w)$,  by another application of the isotopy extension theorem (or, rather its interpretation in terms of the homotopy lifting property of the  map Diff$(M)\rightarrow$Diff$(N,M)$ which evaluates a diffeomorphism at a submanifold, see \cite{Palais,Lima}).  According to the main theorem of  \cite{JTZ}, there is an induced automorphism $(f_\gamma)_*$ of the Floer homology group $\HFa(Y,w)$, and the functoriality of the contact class under pointed contactomorphisms  implies \[(f_\gamma)_*(c(\xi,w))=c(f_{\gamma*}(\xi),w)=c(\eta,w)\]  The automorphism $(f_\gamma)_*$ will, in general, be non-trivial; indeed, Zemke shows that it can be computed via the formula \cite[Theorem D]{ZemkeGraphCob}:
\[ (f_\gamma)_*= \mathrm{Id} + (\Phi_w)_*\circ (A_\gamma)_*\]
where $A_\gamma$ is the chain level map defining the $H_1(Y)/$Tor action on $\HFa(Y,w)$, and $\Phi_w$ is the basepoint action which, in the case of $\CFa(Y,w)$ counts $J$-holomorphic disks which pass through the hypersurface specified by the basepoint exactly once.   The proposition will follow if we can show that contact classes are in the kernel of the $H_1(Y)/$Tor action.   But this is an easy consequence of their definition.  Letting $c\in H_*(\Filt(-Y,K,w,\mathrm{bot}))\cong \F$ denote the generator of the homology of the bottommost non-trivial filtered subcomplex in the filtration of $\CFa(-Y,w)$ induced by the binding of a pointed open book supporting $\xi$,  the contact class is defined as
\[  c(\xi,w):=  \iota_*(c)\]
where $\iota: \Filt(-Y,K,w,\mathrm{bot})\hookrightarrow \CF(-Y,w)$ is the inclusion map.   The chain map $A_\gamma$  on $\CFa(-Y,w)$ respects the filtration induced by $K$, defined as it is by counting $J$-holomorphic disks which avoid $w$ (see \cite[Proof of Proposition 5.8]{Adjunction}).  It follows that $A_\gamma$ maps $\Filt(-Y,K,w,\mathrm{bot})$ to $\Filt(-Y,K,w,\mathrm{bot})$, but since it shifts the relative $\Z/2\Z$ homological grading, and the homology of the latter subcomplex is one-dimensional, the map on homology must be trivial.  Therefore the automorphism on Floer homology induced by a point-pushing map acts as the identity on any contact elements, and we have $c(\xi,w)=(f_\gamma)_*(c(\xi,w))=c(\eta,w)$ as claimed.
\end{proof}

The above proposition can be used to show that change-of-basepoint maps on Floer homology induced by pushing points along arcs act on contact elements in a canonical way i.e.  $(f_\gamma)_*(c(\xi,w))\in \HFa(Y,w')$ is independent from the choice of arc used to construct a diffeomorphism $f_\gamma:(Y,w)\rightarrow (Y,w')$. This indicates an independence of the contact class from the choice of basepoint. We can make this independence more precise.   To do this, we show that the point-pushing maps along arcs can be refined to pointed contactomorphisms.


\begin{prop}{\em(cf. \cite{Hatakeyama})}\label{contactpointpushing}
Given $w,w'\in Y$, there exists a contactomorphism $\phi:(Y,\xi)\to (Y,\xi)$, which is isotopic to the identity and maps $w$ to $w'$.
\end{prop}

\begin{proof}
Let $\gamma:[0,1]\to Y$ denote a smooth embedded path from $w$ to $w'$. After a $\mathcal{C}^\infty$-small isotopy we may assume the path $\gamma$ is transverse to $\xi$. Let $\nu(\gamma([0,1]))$ denote a neighborhood of the transverse arc. A standard neighborhood theorem gives a contact embedding
\[\phi: (\nu (\gamma ([0,1])), \xi)\to (\mathbb{R}^3, ker(\alpha)),\text{ where } \alpha = dz+r^2d\theta, \]
which takes the image of the arc to the segment $\{(0,0)\}\times [0,1]$ along the $z$-axis; in particular $\phi(w) = (0,0,0)$ and $\phi(w')=(0,0,1)$.

Let $\beta$ denote a contact 1-form for $\xi$ which is an extension of $\phi^*\alpha$. The time one flow of the Reeb vector field $R_\beta$ is then the desired contactomorphism taking $w$ to $w'$.
\end{proof}

\begin{cor} The contact class is independent of the basepoint in the following sense: given two  basepoints $w,w'\subset Y$ a path $\gamma$ between them induces an isomorphism $\gamma_*:\HFa(-Y,w)\rightarrow \HFa(-Y,w')$.  For any choice of $\gamma$, the contact class satisfies $\gamma_*(c(\xi,w))=c(\xi,w')$.
\end{cor} 

\begin{proof}
Suppose $\gamma$ is a path connecting $w$ to $w'$. As in the proof of Proposition \ref{prop:istpy}, the homotopy class of $\gamma$ gives rise to a well-defined pointed isotopy class of pointed diffeomorphism $f_\gamma:(Y,w)\to (Y,w')$, whose associated isomorphism between Floer homology groups we denote $\gamma_*$. The contactomorphism constructed using $\gamma$ in the proof of the preceding proposition is a representative of this pointed isotopy class. Functoriality of the invariant under pointed contactomorphisms, Theorem \ref{thm:funccon}, then implies that $\gamma_* (c(\xi,w))=c(\xi,w')$.
\end{proof}

\noindent {\bf Proof of Theorem \ref{thm:basepointindependence}.}  In light of the corollary, if we let the subgroup of $\HFa(-Y,w)$ spanned by contact elements be denoted $cHF(-Y,w)$, we obtain a transitive system (in the sense of \cite[Definition 6.1]{EilenbergSteenrod}) of groups indexed by points in $Y$, for which the isomorphism $f_{w,w'}:cHF(-Y,w)\rightarrow cHF(-Y,w')$ is the map on Floer homology associated to the point-pushing map along {\em any} arc from $w$ to $w'$.  We call the direct limit of this transitive system the {\em contact subgroup} associated to $Y$, and denote it $cHF(-Y)$:
\[ cHF(-Y):= \underset{\longrightarrow}{\mathrm{lim}} \ cHF(-Y,w).\]   The corollary shows that it is well-defined, independent of any choice of basepoint, and that a contact structure $\xi$ on $Y$ receives an associated element $c(\xi)\in cHF(-Y)$ defined as the image of  $c(\xi,w)\in cHF(-Y,w)$  under the canonical inclusion-induced isomorphism $cHF(-Y,w)\rightarrow cHF(-Y)$.   The contact subgroup is functorial with respect to (unpointed) diffeomorphisms of $Y$ by the main theorem of \cite{JTZ}, and an (unpointed)  contactomorphism $f:(Y,\xi)\rightarrow (Y',\xi')$ sends $c(\xi)$ to $c(\xi')$ by Theorem \ref{thm:funccon}.

\section{Functoriality of the contact class under Stein cobordisms}\label{sec:Stein}
In this section we provide a proof of the well-known folk theorem that the \os\ contact invariant is natural with respect to Stein cobordisms. Non-triviality of the contact invariant of a Stein fillable contact structure was proved in \cite[Theorem 1.5]{Contact}.  The proof relied on a naturality result \cite[Theorem 4.2]{Contact} for the  invariants of contact structures represented by open book decompositions which differ by a single Dehn twist.  This latter result implicitly showed that the contact invariant is natural with respect to the Stein cobordisms associated to a Weinstein 2-handle attachment along a Legendrian knot, a fact made more clear  in \cite[Theorem 2.3]{Lisca1} (though stated there in terms of contact $+1$ surgery, and summed over all $\SpinC$ structures on the cobordism) and then explicit in \cite[Proposition 3.3]{Ghiggini1}.  Together with a calculation for $1$-handles, the above results yield a weak naturality of the contact invariant under Stein cobordisms, where one sums over all $\SpinC$ structures.  This is spelled out in \cite[Theorem 11.24]{JuhaszCob}, under an additional  topological restriction on the $1$-handles.  The $\SpinC$ refinement of naturality for the contact invariant of a Stein filling, viewed as a Stein cobordism to the standard structure on the 3-sphere, was established in  \cite[Theorem 4]{OlgaDistinct}.   Given the body of literature on topological aspects of Stein surfaces and domains, exposed beautifully in \cite{EliashbergCieliebakBook,EliashbergStein,Gompf}, essentially the only pieces missing for the $\SpinC$ refinement of naturality under a general Stein cobordism is a discussion of $1$-handles, particularly those with feet in different path components, and a careful discussion of the composition of multiple $2$-handles.   

Recall, then, that a Stein cobordism from a contact $3$-manifold $(Y_1,\xi_1)$ to $(Y_2,\xi_2)$ is a smooth $4$-manifold $W$ with $\partial W=-Y_1\cup Y_2$, oriented by a complex structure $J$ for which the oriented complex lines of tangency on $\partial W$ agree with $\xi_1$ and $\xi_2$, respectively, and which admits a {\em $J$-convex Morse function} $\phi$, defined by the requirement that $-dd^\mathbb{C}\phi=\omega_\phi$ is symplectic.   Such a manifold comes equipped with a Liouville vector field, $X_\phi$, defined as the gradient of $\phi$ with respect to the metric induced by $\omega_\phi$.  See \cite{EliashbergCieliebakBook} for an introduction. 

\begin{theorem}\label{thm:naturality}  Suppose $(W,J,\phi)$ is a Stein cobordism from a contact 3-manifold $(Y_1,\xi_1)$ to a contact 3-manifold $(Y_2,\xi_2)$.  Then \[F_{\Wdual,\mathfrak{k}}(c(Y_2,\xi_2))=c(Y_1,\xi_1),\]
where  $W$ is viewed as a cobordism from $-Y_2$ to $-Y_1$, and $\mathfrak{k}$  is the canonical $\SpinC$ structure associated to $J$.  Moreover 
\[F_{\Wdual,\spinc}(c(Y_2,\xi_2))=0\]
for $\spinc\ne \mathfrak{k}$.
\end{theorem}

\begin{remark} The result is equally valid for Weinstein cobordisms, which \cite{EliashbergCieliebakBook} shows are equivalent to Stein cobordisms for the present purposes.
\end{remark}

\begin{remark} Strictly speaking, the Stein cobordism should be equipped with a properly embedded graph, in the sense of \cite{ZemkeGraphCob}.   In this context, the graph is obtained from the basepoints present on the incoming end of the cobordism by their image under the flow of the Liouville vector field.  We pick basepoints on the incoming ends which flow to the outgoing ends, with  some extra care taken in the case that components of the boundary merge via Stein 1-handles so that all components of the boundary have a single basepoint (see Lemma \ref{lemma:1handle} below).  In light of the naturality results from the previous section, and the resulting independence of the contact class of the choice of basepoint, we can safely omit basepoints from most of the discussion and obtain a naturality result for the contact invariant which is basepoint independent.
\end{remark}

\begin{proof} By \cite[Theorem 1.3.3]{EliashbergStein}, the cobordism can be decomposed as a composition of elementary cobordisms corresponding to Stein 0-, 1-, and 2-handle attachments, with the latter two attached along framed points and Legendrian curves, respectively.  In this dimension, the subtleties involved with $2$-handle framings was clarified by  \cite{Gompf}.  Though we could avoid it with a more cumbersome inductive argument, we can and will assume that the  attachments are ordered by their indices, arising from a self-indexing plurisubharmonic Morse function.  For a smooth manifold, this follows from the standard rearrangement theorem for Morse functions \cite[Theorem 4.8]{Milnor}.  The proof of that theorem, however, modifies the gradient-like vector field for the Morse function so that stable manifold of an index $\lambda$ critical point is disjoint from the unstable manifold of an index $\lambda'\ge \lambda$ critical point (achieving the Morse-Smale condition for the manifolds associated to these critical points).  In the Stein setting, the gradient vector field and metric are coupled, and one cannot vary one without changing the other.  Hence rearranging critical levels is more subtle.   These subtleties are nicely exposed, and dispatched with, in Chapter 10 of \cite{EliashbergCieliebakBook}.  Of particular relevance are Proposition 10.10 and 10.1.  Proposition 10.10 allows one to vary the critical values of the J-convex Morse function specifying the handle decomposition, provided the stable and unstable manifolds of the points of interest are disjoint, the Stein analogue of \cite[Theorem 4.1]{Milnor}.    Proposition 10.1 allows one to vary an isotropic submanifold of a given contact type hypersurface by an isotropic isotopy compatible with a family of J-convex Morse functions, the Stein analogue of \cite[Lemma 4.7]{Milnor}.  Applying the latter to the attaching spheres of the Stein handles allows us to assume, as in the classical case, that the stable manifold of an index $\lambda$ critical point is disjoint from the unstable manifold of an index $\lambda'\ge \lambda$ critical point.  Thus we can proceed by induction to order the handles and further ensure that all critical points of a given index have the same critical value.    The existence of such an ordering for  a 2-dimensional Stein domain (a Stein cobordism with $Y_1=\emptyset$) is implicit in the statement of \cite[Theorem 1.3]{Gompf}, a result which itself is attributed as implicit in Eliashberg \cite{EliashbergStein}.

We assume then, that the cobordism is decomposed as a sequence of elementary 0- and  1-handle cobordisms, followed by a cobordism  associated to a  collection of Weinstein 2-handle attachments along Legendrian curves, equipped with a J-convex Morse function with a unique critical value.   We will show that the contact invariant is mapped in the specified way under a single 0- or 1-handle attachment, and similarly for a simultaneous collection of Stein 2-handle attachments.  The result will then follow from the composition law for cobordism-induced maps on Heegaard Floer homology.   

The fact that the contact invariant is natural under Stein 0-handle attachment follows immediately from the definition of the associated map on Floer homology, which is simply the map induced by the canonical isomorphism between Heegaard Floer chain complexes under taking disjoint union with a Heegaard diagram whose surface is a pointed 2-sphere with  no curves \cite[Section 11.1]{ZemkeGraphCob}.  This definition, together with  the fact that the contact class of the Stein fillable contact structure on the $3$-sphere is non-trivial, yield, upon taking duals, the stated naturality. 
 
We now show that contact invariant is natural with respect to Stein 1-handle cobordisms.

\begin{lemma}\label{lemma:1handle}  Suppose $(W,J)$ is the cobordism associated to a Stein 1-handle attachment.  Then Theorem \ref{thm:naturality} is true for $F_\Wdual$.
\end{lemma}
\begin{proof}  Unlike the case of a Stein domain, a Stein cobordism can have disconnected boundary.   Thus, there are two possibilities (1) the feet of the 1-handle lay in different components of $Y_1$, the incoming boundary of $W$, or (2) the feet of the 1-handle lay in the same component of $Y_1$.  In the former, we may assume without loss of generality that  there only two components of $Y_1$, since Floer homology is manifestly multiplicative under disjoint unions (i.e. groups and homomorphisms associated to disjoint unions of 3-manifolds and cobordisms, respectively, are tensor products), and product cobordisms map contact invariants naturally according to the previous section.\footnote{Here, a product cobordism means a 4-manifold diffeomorphic to $Y\times I$, through a diffeomorphism induced by the flow of the Liouville vector field.  Since the ``holonomy" diffeomorphism from the outgoing boundary to the incoming boundary  \cite[Definition 9.40]{EliashbergCieliebakBook} is a contactomorphism, the naturality results of the previous section indicate the contact invariants are mapped in the specified way. }  Similarly, for the second possibility, we may assume $Y_1$ is connected.

Naturality for 1-handles that connect two components is a consequence of a calculation for the graph cobordism map of the 1-handle cobordism, endowed with a trivalent (strong ribbon) graph that merges the two basepoints in the incoming components to a single basepoint in their outgoing connected sum.  This calculation is the content of \cite[Proposition 5.2]{HMZconnectedsum}, which indicates that such a graph cobordism induces a chain homotopy equivalence, and the complex associated to a connected sum is therefore homotopy equivalent to the tensor product of the complexes associated to the factors. (This calculation reproved, in a functorial way, \ons's earlier connected sum formula \cite[Theorem 6.2]{HolDiskTwo}.)  As the the contact invariant is similarly multiplicative under contact connected sums \cite[Product formula]{tbbounds}, the claimed naturality is immediately implied.   Here, we should point out that the trivalent graph arises naturally from the Stein structure, as the descending manifold of the index one critical point of $\phi$ with respect to $X_\phi$, union a flowline of $X_\phi$ from the critical point to the outgoing boundary.

The case of 1-handles with feet on the same component of the incoming boundary is simpler and, in this case, follows from \os's definition of the 1-handle map \cite[Section 4.3]{HolDiskFour}, together again with the fact that the contact invariant is multiplicative under contact connected sums.  In this case, the outgoing manifold is contactomorphic to the connected sum $(Y\#(S^1\times S^2),\xi\#\xi_{std})$, so it suffices to show that the image of the dual of $c(\xi)$ under  \os's map induced by the 1-handle agrees with the dual $c(\xi\#\xi_{std})$.  But the 1-handle map sends $c(\xi)^*$ to  $c(\xi)^*\otimes \Theta_+$.  Thus the problem is reduced to a single calculation, verifying that the dual of the contact class of the standard contact structure on $S^1\times S^2$ satisfies $\Theta_+=(c(S^1\times S^2,\xi_{std}))^*\in \HFa(S^1\times S^2)$.  This calculation can be done in numerous ways; see,  e.g. \cite[Proof of Proposition 5.19]{Adjunction}, for an explicit treatment. \end{proof}

We now turn to 2-handles.  While the naturality statement given here was proved for the cobordism associated to a {\em single} Stein 2-handle by Ghiggini, this together with naturality under the 0- and 1-handles is not enough to yield the desired result.  Indeed, if $W=W_1\circ W_2$ is a composition of Stein cobordisms, the composition law for cobordism maps states that:

\[ F_{W_1,\mathfrak{k}_1}\circ F_{W_2,\mathfrak{k}_2} = \underset{\{\spinc\in \SpinC(W)\ | \ \spinc|_{W_i}=\mathfrak{k}_i\}}{\sum} F_{W,\spinc}.\]

Assuming naturality of the invariant for $W_1$ and $W_2$ does not, then, imply either of the conclusions in the statement of the theorem when $\SpinC$ structures on $W$ are not uniquely determined by their restrictions to $W_1$ and $W_2$.  The following standard lemma makes this precise

\begin{lemma} Suppose $W=W_1\cup_{Y} W_2$ is a 4-manifold glued along a 3-manifold $Y$ arising as a connected component of $\partial W_i$ (with boundary orientation of $Y$ different for $i=1,2$).  Then the set of Spin$^c$ structures on $W$ restricting to $\spinc_i\in \SpinC(W_i)$, provided it is non-empty, is in affine correspondence with $\delta H^1(Y)$, where $\delta$ is the connecting homomorphism in the Mayer-Vietoris sequence.
\end{lemma} 

In particular, if either $W_i$ is a cobordism associated to a 1-handle attachment, then a $\SpinC$ structure on $W$ is uniquely determined by its restrictions to the pieces.   This is because a 1-handle cobordism has the property that the restriction $H^1(W_i)\rightarrow H^1(Y)$ is surjective, which implies $\delta H^1(Y)$ is trivial, by exactness.    Therefore, we can treat 1-handles individually.  It is certainly possible, however, that a $\SpinC$ structure on a 4-manifold composed of two or more 2-handles will not be determined by its restrictions, and therefore we cannot appeal directly to the naturality of $c(\xi)$ with respect to Stein 2-handles.

The following lemma, together with Ghiggini's argument for a single Stein 2-handle, will be used to establish naturality under Stein cobordisms which are built of multiple 2-handles:

\begin{lemma}{\em (cf. \cite[Theorem 5.3]{OSSymp})} 
\label{lem:lfiso}
 Let $\pi:W\to [0,1]\times S^1$ be a relatively minimal Lefschetz fibration over the annulus, viewed as a cobordism from $Y_1$ to $Y_2$, whose fiber $F$ has genus $g>1$. Then there is a unique Spin$^c$ structure $\mathfrak{s}$ over $W$ for which 
\[
\langle c_1(\mathfrak{s}),[F]\rangle= 2-2g
\]
and the induced map
\[
F_{W,\mathfrak{s}}^+: HF^+ (Y_1, \mathfrak{s}|_{Y_1})\to HF^+ (Y_2, \mathfrak{s}|_{Y_2})
\]
is nontrivial. This is the canonical Spin$^c$ structure $\mathfrak{k}$, and its associated map is an isomorphism.
\end{lemma}
\begin{proof}
Suppose that $\mathfrak{s}\in$ Spin$^c(W)$ is as in the statement of the lemma, and induces a non-trivial map. We will show that $\mathfrak{s}$ is the canonical Spin$^c$ structure $\mathfrak{k}$ on $W$.
By \cite[Theorem 2.2]{OSSymp} $Y_2$ bounds a 4-manifold $W'$ admitting a  Lefschetz fibration over the disk with fiber identified with $F$. Let $V= W\cup W'$. $V$ admits a Lefschetz fibration over the disk obtained by extending $\pi$. 

The composition law for cobordism maps states that 
\begin{equation}
F_{W' - B^4, \mathfrak{k}' }^+\circ F_{W, \mathfrak{s} }^+ = \underset{\{\mathfrak{t} \in \SpinC(W)\ | \ \mathfrak{t}|_{W} = \mathfrak{s},\  \mathfrak{t}|_{W'-B^4} = \mathfrak{k}'\}}{\sum} F_{V-B^4,\mathfrak{t}}^+ 
\end{equation}
where $\mathfrak{k}'$ is the canonical Spin$^c$ structure on $W'$.
By \cite[Theorem 5.3]{OSSymp} the map
\[
F_{W' - B^4, \mathfrak{k}' }^+ : HF^+ (Y_2, \mathfrak{k}'|_{Y_2}) \to HF^+ (S^3)
\]
is an isomorphism.  Similarly, \cite[Theorem 5.3]{OSSymp} implies that there is a unique non-trivial contribution to the sum on the right-hand side, coming from the canonical Spin$^c$ structure $\mathfrak{k}_V$ on $V$, and $F_{V-B^4,\mathfrak{k}_V}^+$ is an isomorphism.  Since $\mathfrak{k}_V$ restricts to the canonical Spin$^c$ structure $\mathfrak{k}$ on $W$, it follows that $\mathfrak{s}=\mathfrak{k}$, and the corresponding map is an isomorphism.
\end{proof}

With this in hand, we modify the argument of \cite[Lemma 2.11]{Ghiggini1} to establish naturality with respect to Stein 2-handles, which boils down to another application of the composition law, together with a theorem of Eliashberg:
\begin{lemma}
Suppose $(W,J)$ is the cobordism associated to a collection of Stein 2-handle attachments. Then Theorem \ref{thm:naturality} is true for $F_{W^{\dagger}}$.
\end{lemma}
\begin{proof}
As detailed above, we assume that all critical points of the plurisubharmonic Morse function on $W$ have the same critical value, or equivalently that $(W,J)$ is constructed by attaching a Stein 2-handle along each component of a Legendrian link $L$ in $(Y_1,\xi_1)$. 

We may choose an open book decomposition adapted to $\xi_1$ such that the Legendrian link $L$ sits naturally in a page. After positively stabilizing the open book we may assume that the pages have connected boundary and are of genus greater than one. 

For $i\in\{1,2\}$ let $V_i$ denote the trace cobordism from $Y_i$ to $Y_i '$, the 3-manifold obtained by performing zero surgery along the binding of the open book. Surgery along $L$ gives rise to the cobordism $W$ from $Y_1$ to $Y_2$, and a cobordism $W_0$ from $Y_1'$ to $Y_2'$. Note that both $Y_1'$ and $Y_2'$ are fibered 3-manifolds with fiber $F$ obtained by capping off the boundary component of a page of the open book.
$W_0$ admits a Lefschetz fibration over the annulus with fiber $F$.

Let $X = W\cup V_2 \cong V_1\cup W_0$ denote the cobordism from $Y_1$ to $Y_2'$. By
\cite[Theorem 1.1]{EliashbergFilling} we may extend the symplectic structure 
induced by the Lefschetz fibration on $W_0$ over the 2-handle cobordism $V_1$, giving a symplectic structure $\omega$ on $X$. The restriction of $\omega$ to $W$ agrees with the symplectic structure on $W$ induced by the Legendrian surgery along $L$; in particular the canonical Spin$^c$ structure $\mathfrak{k}_X\in $ Spin$^c(X)$ of $\omega$ restricts to the canonical Spin$^c$-structure $\mathfrak{k}$ of $(W,J)$.

Since $V_1$ can be obtained from surgery along a homologically non-trivial curve in $Y_1'$, restriction induces an isomorphism $H^2(X,\mathbb{Z})\to H^2(W_0,\mathbb{Z})$, so every Spin$^c$-structure on $W_0$ admits a unique extension over $X$. In particular, the extension of the canonical Spin$^c$-structure $\mathfrak{k}_0\in$  Spin$^c (W_0)$ is $\mathfrak{k}_X$. For $i\in \{1,2\}$, let $\mathfrak{t}_i =\mathfrak{k}_0 |_{Y_i '}$.

Ozsv\'ath and Szab\'o \cite{Contact} characterize the contact invariant $c(\xi_i)\in HF^+ (-Y_i, \mathfrak{s}_{\xi _i})$ as the image of a class $c(\pi_i)\in HF^+(-Y_i ',\mathfrak{t}_i )$ associated to the fibration under the map $F^+ _{V_i ^\dagger,\mathfrak{p}_i}$ where $\mathfrak{p}_i\in$ Spin$^c(V_i)$ is the unique extension of $\mathfrak{t}_i$. Let $\mathfrak{s}\in$ Spin$^c(W)$, then

\[
F_{W^{\dagger},\mathfrak{s}}^+ (c(\xi_2)) = F_{W^{\dagger},\mathfrak{s}}^+ \circ F^+_{V_2 ^{\dagger}}(c(\pi_2)) = \underset{\{\mathfrak{s}_X \in \SpinC(X)\ | \ \mathfrak{s}_X|_{W} = \mathfrak{s},\  \mathfrak{s}_X|_{V_2} = \mathfrak{p}_2\}}{\sum} F_{X^{\dagger},\mathfrak{s}_X}^+ (c(\pi_2)),
\]
where the last equality is given by the composition law for cobordism maps. Because every Spin$^c$-structure on $W_0$ admits a unique extension over $X$, another application of the composition law gives that the above sum is equal to
\[
\underset{\{\mathfrak{s}_X \in \SpinC(X)\ | \ \mathfrak{s}_X|_{W} = \mathfrak{s},\  \mathfrak{s}_X|_{V_2} = \mathfrak{p}_2\}}{\sum} F_{V_1 ^{\dagger}, \mathfrak{s}_X|_{V_1}}^+\circ F_{W_0^{\dagger},\mathfrak{s}_X|_{W_0}}^+ (c(\pi_2)).
\]

Note that $\langle c_1(\mathfrak{s}_X|_{W_0}),[F]\rangle=\langle c_1(\mathfrak{t}_2),[F]\rangle=2-2g$. Lemma \ref{lem:lfiso} implies that there is at most one non-zero contribution to this sum, coming from a term where $\mathfrak{s}_X |_{W_0}$ is the canonical Spin$^c$-structure $\mathfrak{k}_0$, in which case $\mathfrak{s}_X = \mathfrak{k}_X$ and $\mathfrak{s} = \mathfrak{k}$ by the preceding discussion. We have
\[ F_{W^{\dagger},\mathfrak{s}}^+ (c(\xi_2)) = \begin{cases} 
      F^+_{V_1 ^{\dagger}}\circ F_{W_0 ^{\dagger},\mathfrak{k}_0}^+(c(\pi_2)) & \text{ for }\mathfrak{s} =\mathfrak{k} \\
      0 & \text{ otherwise.}
   \end{cases}
\]
Moreover, Lemma \ref{lem:lfiso} also tells us that $F_{W_0 ^{\dagger},\mathfrak{k}_0}^+$ is an isomorphism mapping $c(\pi_2)$ to $c(\pi_1)$, thus
\[F_{W^{\dagger},\mathfrak{s}}^+ (c(\xi_2)) = \begin{cases} 
      F^+_{V_1 ^{\dagger}}(c(\pi_1)) = c(\xi_1) & \text{ for }\mathfrak{s} =\mathfrak{k} \\
      0 & \text{ otherwise.}
   \end{cases}
\]
\end{proof}
\noindent Having established the claimed naturality result  for a Stein 0- or 1-handle and for a collection of Stein 2-handles, the theorem follows from the composition law for cobordism maps.
\end{proof}

\begin{remark} Echoing Remark \ref{rem:HKMinvariance}, one could alternatively approach Theorem \ref{thm:naturality} using the Honda-Kazez-Mati{\'c} interpretation of the contact invariant.  
Using their Heegaard diagrams, the proof hinges on (a) showing that there exists a unique pseudo-holomorphic triangle contributing to  $F_{W^{\dagger}}^+ (c(\xi_2))$ whose domain is a union of small triangles having corners at the components of a generator representing $c(\xi_1)$ and (b)
identifying the Spin$^c$ structure associated to this pseudo-holomorphic triangle with the canonical one.

\end{remark}

\noindent We conclude with the following immediate corollary of Theorem \ref{thm:naturality}, which generalizes the main result of \cite{OlgaDistinct}:

\begin{cor}{\em(cf. \cite[Theorem 2]{OlgaDistinct}) }\label{olgageneralization} Let $W$ be a smooth 4-manifold with boundary, equipped with two Stein structures $J_1,J_2$ with associated {\em Spin}$^c$ structures $\spinc_1,\spinc_2$, and let $\xi_1,\xi_2$ be the induced contact structures on $Y$, the outgoing boundary of $W$. If the {\em Spin}$^c$ structures $\spinc_1$ and $\spinc_2$ are not isomorphic, then the contact invariants $c(\xi_1), c(\xi_2)$ are distinct elements of $\HFa(-Y)$.
\end{cor}

\bibliographystyle{plain}
\bibliography{SteinCobC}

\end{document}